\documentclass[11pt,reqno]{amsart}
\usepackage{latexsym}
\RequirePackage{amsmath}
\RequirePackage{amssymb}
\RequirePackage{amsthm}
\RequirePackage{epsfig}
\newtheorem{theorem}{Theorem}
\newtheorem{theoremO}{Theorem}

\newtheorem*{corollary*}{Corollary}

\newcommand {\SC} {{\mathbb C}}  \newcommand {\SD} {{\mathbb D}}  \newcommand {\SN} {{\mathbb N}}
\newcommand {\SR} {{\mathbb R}}  \newcommand {\ST} {{\mathbb T}}

  \newcommand{\vt}{\vartheta}  \newcommand{\la}{\lambda}
  \newcommand{\vp}{\varphi}  
  
\newcommand{\cP}{{\mathcal P}}

\newcommand{\be}{\begin{equation}}
\newcommand{\ee}{\end{equation}}
\newcommand{\bea}{\begin{eqnarray}}
\newcommand{\eea}{\end{eqnarray}}

\begin{document}
\author[I. Efraimidis]{Iason Efraimidis}
\address{Departamento de Matem\'aticas, Universidad Aut\'onoma de
Madrid, 28049 Madrid, Spain}
\email{iason.efraimidis@uam.es}
\date{November 6, 2015}

\title[On coefficient inequalities for functions with positive real part]{A generalization of Livingston's coefficient inequalities for functions with positive real part}

\begin{abstract}
For functions $p(z) = 1 + \sum_{n=1}^\infty p_n z^n$ holomorphic in the unit disk, satisfying $ {\rm Re}\, p(z) > 0$, we generalize two inequalities proved by Livingston \cite{Li69, Li85} and simplify their proofs. One of our results states that $|p_n -w p_k p_{n-k}|\leq 2\max\{1, |1-2w|\}, w\in\SC$. Another result involves certain determinants whose entries are the coef\mbox{}ficients $p_n$. Both results are sharp. As applications we provide a simple proof of a theorem of Brown \cite{B10} and various inequalities for the coefficients of holomorphic self-maps of the unit disk.  
\end{abstract}
\maketitle

\section{Introduction}

Let $\cP$ denote the class of functions of the form $p(z) = 1 + p_1 z + p_2 z^2 + \ldots$ which are analytic in the unit disk $\SD = \{z\in\SC : |z|<1\}$ and satisfy $ {\rm Re}\, p(z) > 0$ for all $z \in \SD$. As early as 1911 Carath\'eodory proved that coefficients of functions in $\cP$ satisfy $|p_n|\leq 2$ (Theorem \ref{TA} below). Livingston \cite{Li69} proved that \mbox{$|p_n - p_k p_{n-k}|\leq2$,} for all $0\leq k\leq n$. He used this inequality in his study of multivalent close-to-convex functions, while more applications were later found in \cite{BT}, \cite{LiZl} and \cite{M2}. In this note we generalize this inequality by finding the sharp bound of \mbox{$|p_n -w p_k p_{n-k}|$,} $w\in\SC$, in Theorem \ref{T1}. In particular, the bound $2$ is still valid whenever $w$ lies in the disk $\{w:|1-2w|\leq 1\}$. 

For $w\in\SC$ and $p\in\cP$ we define the $(k+1)\times(k+1)$ determinant
\begin{equation*}
A_{k,n}(w)= \left|
\begin{array}{cccccc}
p_{n+k} & p_{n+k-1} & p_{n+k-2} &  \ldots & p_{n+1} & p_n   \\
wp_1     & 1               & 0               &  \ldots & 0            & 0       \\
wp_2     & wp_1         & 1               &  \ldots & 0            & 0       \\
\vdots    & \vdots        &   \vdots      & \ddots & \vdots     & \vdots\\
wp_{k-1}& wp_{k-2}   & wp_{k-3}   &  \ldots & 1            &   0      \\
wp_k     & wp_{k-1}   & wp_{k-2}    &  \ldots & wp_1      &   1
\end{array} \right|.
\end{equation*}
Livingston \cite{Li85} defined this for $w=1$ and proved that $|A_{k,n}(1)|\leq 2$. In Theorem \ref{T2} we find the sharp bound of $|A_{k,n}(w)|$ for all $w\in\SC$. When no confusion arises we will suppress $w$ and write $A_{k,n}$ for $A_{k,n}(w)$. Here are some examples of initial $A_{k,n}$'s: 
\begin{align*}
A_{0,n} &= \, p_n, \qquad A_{1,n} = p_{n+1}-wp_1p_n, \\ \qquad\;
A_{2,n} &= \, p_{n+2} - wp_1p_{n+1} - wp_2p_n +w^2 p_1^2 p_n.
\end{align*}

In order to fix the notation let $n\in\SN$ and denote by $U_n = \{ e^{2k\pi i/n} : k=1,2,\ldots,n\}$ the set of $n$-th roots of unity. For $n=0$ we understand $U_0$ as $\ST= \partial\SD$. Also, for a set $E\subset \SC$ and a number $a\in\SC$ we write $aE=\{az : z\in E\}$.

The Herglotz representation \cite[p.22]{Du} asserts that for every $p \in \cP$ there is a unique probability measure $\mu$ supported on $\ST$, such that 
$$
p(z) = \int_\ST \frac{1+\la z}{1-\la z} d\mu (\la), \quad z \in \SD.
$$
We call $\mu$ the \emph{Herglotz measure} of $p$ and write $supp(\mu)$ for its support. One can readily see that the coefficients satisfy $p_n = 2 \int_\ST \la^n d\mu(\la)$. 

We now state Carath\'eodory's Theorem \cite[p.41]{Pom} and our two main theorems \ref{T1} and \ref{T2}.

\begin{theoremO} \label{TA}
If $p\in\cP$ then $|p_n|\leq 2$ for all $n\geq1$. For a fixed $n$, equality holds if and only if $supp(\mu) \subseteq e^{i\vp} U_n$ for some $\vp \in [0,2\pi)$.
\end{theoremO}

\begin{theorem} \label{T1}
If $p\in\cP$ and $w\in \SC$ then
$$
|p_n - w p_k p_{n-k}| \leq 2\max\{1, |1-2w|\}
$$
 for all $1\leq k \leq n-1$.

Let $\mu$ be the Herglotz measure of $p$. In the case $|1-2w|<1$, equality holds if and only if $p_k=0$ and $supp(\mu) \subseteq e^{i\vp} U_n,$ for some $\vp \in [0,2\pi)$. In the case $|1-2w|>1$, equality holds if and only if $supp(\mu) \subseteq e^{i\vt} U_k \cap e^{i\vp} U_n$, for some $\vt,\vp \in [0,2\pi)$. In the case $|1-2w|=1$, if $supp(\mu)$ consists of one point then equality holds.
\end{theorem}

\begin{theorem} \label{T2}
If $p\in\cP$ and $w\in\SC$ then
$$
|A_{k,n}(w)|\leq 2\max\{1, |1-2w|^k\}
$$
for all $k\geq0$ and $n\geq1$.

Let $\mu$ be the Herglotz measure of $p$. In the case $|1-2w|<1$, equality holds if and only if $supp(\mu) \subseteq e^{i\vp} U_{n+k},$ for some $\vp \in [0,2\pi)$ and $p_1=p_2=\ldots=p_k=0$. In the case $|1-2w|\geq1$, if $supp(\mu)$ consists of one point then equality holds.
\end{theorem}

The condition for equality in Theorem \ref{T1}, in the case $|1-2w|=1$, is far from being necessary. To illustrate this consider $w=1$, $n=2k$ and a Herglotz measure supported on two arbitrary points $\la_1,\la_2$ on $\ST$ having equal point masses, $1/2$ each. Then the coefficients of the corresponding function in $\cP$ are $p_j=\la_1^j+\la_2^j$ and one easily computes
$$
|p_{2k}-p_k^2| = |\la_1^{2k}+\la_2^{2k}-(\la_1^k+\la_2^k)^2| = 2.
$$
The complete characterization of equality when $|1-2w|=1$ is given in Theorem~\ref{T3}. Note that in the special case where $w=1$, the form of the extremal functions was not explicitly stated in \cite{Li69}.

Since the set $e^{i\vt} U_k \cap e^{i\vp} U_n$ in Theorem \ref{T1} cannot be empty, as $supp(\mu) \neq \emptyset$, the number of points it contains is equal to the greatest common divisor of $k$ and $n$.

Both Theorems \ref{T1} and \ref{T2} have a version for non-normalized functions $p(z) = \sum_{n=0}^\infty p_n z^n$ with positive real part. For such a function $p$, let $p_0=x+iy, (x>0)$ and $q(z)=\big(p(z)-iy\big)/x$, which is obviously a function in $\cP$. To this $q$, having coefficients $q_n=p_n/x$, we can apply Theorems \ref{T1} and \ref{T2}. Then multiply both inequalities by $x/|p_0|$ and set $wx/p_0$ in place of $w$. What results is
$$
\left| \frac{p_n}{p_0} - w \frac{p_k p_{n-k}}{p_0^2}\right| \leq 2 \frac{{\rm Re}\,p_0}{|p_0|} \max\left\{1, \left|1-\frac{2w{\rm Re}\,p_0}{p_0}\right|\right\}
$$
and 
$$
|A_{k,n}|\leq 2\frac{{\rm Re}\,p_0}{|p_0|}\max\left\{1, \left|1-\frac{2w{\rm Re}\,p_0}{p_0}\right|^k\right\}
$$
for the modified $A_{k,n}$, having $p_j/p_0$ in place of $p_j$ (for all $j$). Note that for $w=1$ the two entries in the maximum are equal and what one gets is Livingston's original results. 

An alternative proof for the inequality in Theorem \ref{T2} under the additional condition $n\geq k+1$ can be given via the method of Delsarte and Genin \cite{DeGe}. Their approach relies on the observation that $A_{k,n}(1)$ is related to a truncation of the reciprocal of a function in $\cP$. With the aid of Herglotz' formula they get a substantially simpler proof of Livingston's result. The proof, which will be presented in section 2, is an adaptation of their arguments to our case of $A_{k,n}(w)$, for any $w\in\SC$.

Finally, we turn to a question raised by Goodman (\cite[p.104]{Go83}) about the sharp bound of $|p_{n+1}-p_n|$ for functions in $\cP$ with prescribed $p_1$. Using extreme point theory, Brown \cite{B10} proves the following theorem.

\begin{theoremO}  \label{TB}
Let $p(z)=1+\sum_{k=1}^\infty p_k z^k$ be in $\cP$, $m,n\in\SN$ and $\nu\in\SR$. Then
$$
|e^{i\nu}p_{n+m}-p_n| \leq 2 \sqrt{2- {\rm Re}\, (e^{i\nu}p_m)}.
$$
The result is sharp.
\end{theoremO}
In this note we provide a simpler proof. 

We proceed with section 2 where the proofs of Theorems \ref{T1}, \ref{T2} and \ref{TB} are presented. In section 3 we carry out a detailed study of the equality case in a special case of Theorem \ref{T1}, namely the case $|1-2w|=1$. In section 4 we deduce a simple corollary of Theorems \ref{T1} and \ref{T2} for initial coefficients of self-maps of $\SD$ that fix the origin. 

\section{Proofs of Theorems \ref{T1}, \ref{T2} and \ref{TB}}

\begin{proof}[Proof of Theorem \ref{T1}]
First we note that $|1-2w|\leq 1$ if and only if $|w|^2 \leq {\rm Re}\, w$. We compute 
\begin{align*}
|p_n - w p_k p_{n-k}| &= \left| 2 \int_\ST \la^n d\mu(\la) - 2 w p_k \int_\ST \la^{n-k} d\mu(\la) \right| \\
& \leq 2 \int_\ST | \la^n - w p_k \la^{n-k} | d\mu(\la) \\
& \leq \; 2 \left( \int_\ST | \la^k - w p_k |^2 d\mu(\la) \right)^{1/2} \\
& = 2 \left( \int_\ST 1 - 2 {\rm Re}\,( w p_k\la^{-k}) + |w p_k|^2 d\mu(\la) \right)^{1/2} \\
& = 2 \Big( 1 - 2 {\rm Re}\, (w p_k\overline{p_k}/2) + |w p_k|^2 \Big)^{1/2} \\
& = 2 \Big(1 + (|w|^2 - {\rm Re}\, w)|p_k|^2\Big)^{1/2} \\
&\leq 2\max\{1, |1-2w|\}. 
\end{align*}
Here we used the triangle and Cauchy-Schwarz inequalities. At the last step, in the case $|1-2w|>1$, we made use of Theorem \ref{TA}. 

Now suppose that equality holds. If $|1-2w|<1$ then equality in the last of the above inequalities yields $p_k=0$. Hence the second term in $p_n - w p_k p_{n-k}$ vanishes and we have $|p_n|=2$. By Theorem \ref{TA}, $supp(\mu) \subseteq e^{i\vp} U_n$ for some $\vp \in [0,2\pi)$. 

In the case $|1-2w|>1$, the last inequality yields $|p_k|=2$. Hence $supp(\mu) \subseteq e^{i\vt} U_k$ for some $\vt \in [0,2\pi)$. Now $p_k=2e^{ik\vt}$ and 
$$
p_{n-k} = 2 \int_\ST \la^{n-k} d\mu(\la) = 2 e^{-ik\vt} \int_\ST \la^n d\mu(\la) = e^{-ik\vt}p_n.
$$
Hence $2|1-2w|=|p_n - 2 w p_n|$, which implies that $|p_n|=2$. Again by Theorem \ref{TA} we have $supp(\mu) \subseteq e^{i\vp} U_n$ for some $\vp \in [0,2\pi)$ and thus $supp(\mu)$ must form a subset of the intersection $e^{i\vt} U_k \cap e^{i\vp} U_n$.

It is elementary to check that in all three cases the conditions are sufficient for equality.
\end{proof}
\vspace{0,2cm}

\begin{proof}[Proof of Theorem \ref{T2}]
Let $n\geq1$ and $w\in\SC$ be fixed. The case $k=0$ follows from Theorem \ref{TA}. For $k\geq1$ we define
\begin{equation*}
Q_{k,n}(\la)= \left|
\begin{array}{ccccc}
\la^{n+k-1}& p_{n+k-1} & p_{n+k-2} & \ldots & p_n \\
w               & 1               & 0               & \ldots & 0     \\
w\la           & wp_1         & 1               & \ldots & 0     \\

\vdots        & \vdots        &   \vdots       &\ddots & \vdots\\
w\la^{k-1} & wp_{k-1}    & wp_{k-2}    & \ldots & 1
\end{array}  \right|.
\end{equation*}
Expanding $A_{k,n}$ along the first column, using the Herglotz formula and the linearity of the integral, and finally putting the determinant back together, we get $A_{k,n}=2\int_{\ST}\la Q_{k,n}(\la) d\mu(\la)$. 

We will now show by induction that
\be \label{ind}
\int_{\ST}|Q_{k,n}(\la)|^2 d\mu(\la) \leq \max \{ 1, |1-2w|^{2k} \}
\ee
for all $k\geq1$. Then the desired inequality will follow since
\begin{align}\label{ineq}
|A_{k,n}| &\leq 2\int_{\ST}|Q_{k,n}(\la)| d\mu(\la) \nonumber \\ 
&\leq 2 \left(\int_{\ST}|Q_{k,n}(\la)|^2 d\mu(\la) \right)^{1/2} \nonumber \\
&\leq 2\max\{1, |1-2w|^k\},
\end{align}
by the triangle and Cauchy-Schwarz inequalities. 

We first prove \eqref{ind} for $k=1$. (Recall that $|1-2w|<1$ if\mbox f $|w|^2<{\rm Re}\, w$.)
\begin{align}\label{j=1}
\int_{\ST}|Q_{1,n}(\la)|^2 d\mu(\la) =& \int_{\ST} 1 +|wp_n|^2 -2{\rm Re}\, (w p_n \la^{-n}) d\mu(\la) \nonumber \\
=   & 1+ (|w|^2-{\rm Re}\, w)|p_n|^2 \\
\leq &  \max \{ 1, |1-2w|^2 \}. \nonumber
\end{align}
Next, let us assume that \eqref{ind} holds for $k$ and let us prove it for $k+1$ instead of $k$. Expanding $Q_{j,n}(\la)$ along the second row it is not difficult to see that
\begin{align*}
Q_{j,n}(\la)=& \left|
\begin{array}{cccc}
\la^{n+j-1}& p_{n+j-2}  & \ldots & p_n \\
w\la            & 1               & \ldots & 0     \\
w\la^2        & wp_1         & \ldots & 0     \\
\vdots         & \vdots          &\ddots & \vdots\\
w\la^{j-1}  & wp_{j-2}    & \ldots & 1
\end{array}  \right| -w\left|
\begin{array}{cccc}
p_{n+j-1}  & p_{n+j-2} & \ldots  & p_n \\
wp_1              & 1               & \ldots  & 0     \\
wp_2          & wp_1         & \ldots  & 0     \\
\vdots             & \vdots        &\ddots  & \vdots\\
wp_{j-1}    & wp_{j-2}    & \ldots  & 1
\end{array}  \right|\\
=&\la Q_{j-1,n}(\la) -w A_{j-1,n}.
\end{align*}
For $j=k+1$ we have $Q_{k+1,n}(\la)=\la Q_{k,n}(\la) -w A_{k,n}$. Hence
\begin{align}
\int_{\ST}&|Q_{k+1,n}(\la)|^2 d\mu(\la) =\nonumber\\
&= \int_{\ST} \left[ |Q_{k,n}(\la)|^2  -2{\rm Re}\,\left(w \overline{\la Q_{k,n}(\la)} A_{k,n} \right)\right] d\mu(\la)  +|w|^2|A_{k,n}|^2 \nonumber \\
&= \int_{\ST}  |Q_{k,n}(\la)|^2 d\mu(\la) +(|w|^2-{\rm Re}\,w) |A_{k,n}|^2. \label{Q}
\end{align}
We distinguish two cases. If $|1-2w|<1$ then \eqref{Q} and \eqref{ind} show that 
\be
\int_{\ST}|Q_{k+1,n}(\la)|^2 d\mu(\la) \leq \int_{\ST}|Q_{k,n}(\la)|^2 d\mu(\la)  \leq 1. \label{ind-<}
\ee
For the case $|1-2w|\geq 1$, we make a further use of the Cauchy-Schwarz inequality to obtain $|A_{k,n}|^2 \leq 4\int_{\ST}|Q_{k,n}(\la)|^2 d\mu(\la)$. Now by \eqref{Q} and \eqref{ind} we get that 
\begin{align*}
\int_{\ST}|Q_{k+1,n}(\la)|^2 d\mu(\la) &\leq (1 +4|w|^2-4{\rm Re}\,w) \int_{\ST}|Q_{k,n}(\la)|^2 d\mu(\la) \\
&\leq |1-2w|^2 |1-2w|^{2k} \\
&=|1-2w|^{2k+2}.
\end{align*}
Hence \eqref{ind} has been proved for all $k\geq1$.

We now turn to the case of equality. Suppose that $|1-2w|<1$ and $|A_{k,n}|=2$. Then inequalities \eqref{ineq} become equalities and in particular $\int_{\ST}  |Q_{k,n}(\la)|^2 d\mu(\la)=1$. The inductive step \eqref{ind-<} shows that $\int_{\ST}  |Q_{j,n}(\la)|^2 d\mu(\la)=1$ for all $j=1,2,\ldots,k$. This is true in particular for $j=1$, which by \eqref{j=1} implies that $p_n=0$. This in turn is easily seen to imply that $A_{k,n}=A_{k-1,n+1}$. Hence we may repeat the above argument to get that $\int_{\ST}  |Q_{j,n+1}(\la)|^2 d\mu(\la)=1$ for all $j=1,2,\ldots,k-1$. Again from $j=1$ we get by \eqref{j=1} that $p_{n+1}=0$. We repeat this argument until we get $p_n=p_{n+1}=\ldots=p_{n+k-1}=0$. Now $A_{k,n}=A_{0,n+k}=p_{n+k}$ is a number of modulus 2 and therefore Theorem \ref{TA} yields $supp(\mu) \subseteq e^{i\vp} U_{n+k},$ for some $\vp \in [0,2\pi)$. Finally, for all $j=1,2,\ldots,k$ we have
$$
p_j = 2\int_{\ST}  \la^j d\mu(\la) = 2 e^{i(n+k)\vp} \int_{\ST}  \la^{j-n-k} d\mu(\la) =  e^{i(n+k)\vp} \overline{p_{n+k-j}} = 0. 
$$

In both cases the sufficiency for equality is easy to verify.
\end{proof}
\vspace{0,2cm}

\begin{proof}[Alternative proof of Theorem \ref{T2} (case $n\geq k+1$)] Let $w\in\SC$ be fixed. The case $k=0$ follows from Theorem \ref{TA}. Let $k\geq1$ and consider the perturbation
$$
p^*(z) = 1 +w(p_1z+\ldots+p_kz^k)+p_{k+1}z^{k+1}+\ldots
$$
Let $Q_k(z) = 1+q_1z+\ldots +q_kz^k$ be the $k^{th}$ partial sum of $(p^*)^{-1}$, the reciprocal of $p^*$. We define $v_k(z) = \sum_{m=0}^\infty v_{k,m}z^m$, analytic at the origin, via the identity
$$ 
Q_k(z) p^*(z) = 1 + 2 z^{k+1} v_k(z).
$$
Computing the coefficient of $z^{k+m+1}$, for $m\geq k$, we get that
\be\label{v1}
2 v_{k,m} = \sum_{j=0}^k q_j \, p_{k+m+1-j}.
\ee
Note that for $k_1\neq k_2$ the coefficients $q_j$ coincide for $1\leq j \leq\min\{k_1, k_2\}$, hence formula \eqref{v1} readily implies that 
\be\label{v2}
2v_{k,m} = q_{k}p_{m+1} +2 v_{k-1,m+1}.
\ee 
We now proceed with induction on $k\geq1$ to prove that 
\be \label{v=A}
2v_{k,m} = A_{k,m+1}(w), \quad \text{for all} \quad m\geq k.
\ee

For $k=1$ it is easy to verify that $2v_{1,m} = p_{m+2}-wp_1p_{m+1}  = A_{1,m+1}$ for all $m\geq 1$.

Next we suppose that \eqref{v=A} holds for some $k$. We shall prove it for $k+1$ instead of $k$. Expanding with respect to the last column we see that 
\begin{align*}
A_{k+1,m+1} &= A_{k,m+2} + p_{m+1} (-1)^{k+1} \left|
\begin{array}{cccccc}
wp_1        & 1                 &  \ldots  & 0       \\
wp_2        & wp_1           &  \ldots  & 0       \\
\vdots       & \vdots          & \ddots   & \vdots\\
wp_k        & wp_{k-1}      &  \ldots  &  1      \\
wp_{k+1} & wp_k            &  \ldots  &   wp_1
\end{array} \right|\\
 &= A_{k,m+2} + p_{m+1} q_{k+1},
\end{align*}
where we made use of Wronski's formula \cite[p.17]{He74} for the coefficients of the reciprocal of a power series. Therefore by \eqref{v2} we get that
$$
A_{k+1,m+1} = 2v_{k,m+1} + p_{m+1} q_{k+1} = 2v_{k+1,m}
$$
for $m\geq k+1$. Thus \eqref{v=A} has been proved. We set $m=n-1$ and write $A_{k,n}(w) = 2v_{k,n-1}$, for $n\geq k+1$.

We proceed as in \cite{DeGe} using the Herglotz formula in \eqref{v1} and the Cauchy-Schwarz inequality to get
$$
|v_{k,n-1}|^2 = \left| \int_{\ST}\la^{k+n} Q_k(\,\overline{\la}\,) d\mu(\la) \right|^2 \leq \int_{\ST}|Q_k(\,\overline{\la}\,)|^2 d\mu(\la).
$$
Now, we show that 
\be \label{ind2}
\int_{\ST}|Q_k(\,\overline{\la}\,)|^2 d\mu(\la) \leq \max\{1, |1-2w|^{2k}\}
\ee
by induction on $k\geq 1$.

For $k=1$ we compute 
$$
\int_{\ST}|Q_1(\,\overline{\la}\,)|^2 d\mu(\la) = 1+|p_1|^2(|w|^2-{\rm Re}\, w) \leq \max\{1, |1-2w|^2\}.
$$

Now we suppose that \eqref{ind2} is true for $k$. We shall prove it for $k+1$ instead of $k$. We compute
\begin{align*}
\int_{\ST}|Q_{k+1}(\,\overline{\la}\,)|^2 d\mu(\la) &=  \int_{\ST} \sum_{j,m=0}^{k+1}q_j \overline{q_m} \la^{m-j} d\mu(\la) \\
& = \sum_{j=0}^{k+1} |q_j|^2 + {\rm Re}\, \left( \sum_{j<m} q_j \overline{q_m} p_{m-j}  \right),
\end{align*}
where $j=0,1,\ldots,k$ and $m=1,2,\ldots,k+1$ at the last summation. Therefore
\begin{align*}
\int_{\ST}|Q_{k+1}(\,\overline{\la}\,)|^2 d\mu(\la) &= \int_{\ST}|Q_k(\,\overline{\la}\,)|^2 d\mu(\la) + |q_{k+1}|^2  + {\rm Re}\, \left( \overline{q_{k+1}} \sum_{j=0}^k q_j  p_{k+1-j}  \right)\\
 & = \int_{\ST}|Q_k(\,\overline{\la}\,)|^2 d\mu(\la) + (|w|^2-{\rm Re}\, w) \left| \sum_{j=0}^k q_j  p_{k+1-j} \right|^2,
\end{align*}
since $q_{k+1} + w \sum_{j=0}^k q_j  p_{k+1-j} =0$ by the definition of $Q_{k+1}$. If $|1-2w|<1$ then  
$$
\int_{\ST}|Q_{k+1}(\,\overline{\la}\,)|^2 d\mu(\la) \leq \int_{\ST}|Q_k(\,\overline{\la}\,)|^2 d\mu(\la)  \leq 1
$$
and we are done. If  $|1-2w|\geq1$ then we make a further use of the Herglotz formula to get
$$
\left| \sum_{j=0}^k q_j  p_{k+1-j} \right|^2 =\left| 2 \int_{\ST} \la^{k+1} Q_k(\,\overline{\la}\,) d\mu(\la)  \right|^2 \leq 4 \int_{\ST}|Q_k(\,\overline{\la}\,)|^2 d\mu(\la).
$$
Hence
$$
\int_{\ST}|Q_{k+1}(\,\overline{\la}\,)|^2 d\mu(\la) \leq  (1+4|w|^2 -4{\rm Re}\, w) \int_{\ST}|Q_k(\,\overline{\la}\,)|^2 d\mu(\la)  \leq |1-2w|^{2k+2}
$$
and \eqref{ind2} has been established.

It is not clear how one can make the above argument work when $n\leq k$. 
\end{proof}
\vspace{0,05cm}

\begin{proof}[Proof of Theorem \ref{TB}.]
The proof relies on a further generalization of Theorem \ref{T1}. Let $w\in\SC$ and compute
\begin{align*}
|p_{n+m}-wp_n|& \leq 2 \int_\ST |\la^m -w | d\mu(\la) \\
& \leq \; 2 \left( \int_\ST |\la^m -w |^2 d\mu(\la) \right)^{1/2} \\
& = 2 \Big( 1+|w|^2 - {\rm Re}\, (\overline{w} p_m) \Big)^{1/2}. 
\end{align*}
Choosing $w=e^{-i\nu}$ we obtain the desired inequality. Equality evidently holds for the half-plane function $\frac{1+z}{1-z}$.
\end{proof}
\vspace{0,05cm}

\section{Case of equality for Theorem \ref{T1}}

We now consider the case of equality for Theorem \ref{T1} when $|1-2w|=1$. Since our result is more general than Livingston's, it is not surprising that the conditions for equality and their proofs are lengthy.

\begin{theorem} \label{T3}
Let $p\in\cP$, $\mu$ be its representing Herglotz measure, $1\leq k \leq n-1$ and $w=( 1 +e^{i\vt} )/2$ with $|\vt|<\pi$. Then $p_n - w p_k p_{n-k} = 2e^{ic}$ for some $c$ in $[0,2\pi)$ if and only if either
\smallskip

{\rm (i)} $\;p_k=0$ and $supp(\mu) \subseteq e^{i c/n} U_n$; $\,$ or
 \smallskip

{\rm (ii)} $p_k\neq0$ and
\be \label{stat}
supp(\mu) \subseteq (e^{i\frac{\psi}{n-2k}} U_{n-2k}\cap e^{i\left(\frac{\vp}{k}+\frac{c-\psi}{2k} \right)} U_k) \cup (e^{i\frac{\psi}{n-2k}} U_{n-2k}\cap e^{i\left(\frac{\pi-\vp}{k} +\frac{c-\psi}{2k}\right)} U_k)
\ee
for some $\psi$ in $[0,2\pi)$ and $|\vp|\leq \pi/2$. Except for the degenerate case where the support of $\mu$ consists of only one point, the total mass of the measure in each of the two sets of the union is (respectively) equal to 
$$
\frac{1}{2} \left(1 + \frac{\sin\vt}{1+\cos\vt} \tan\vp \right) \quad \text{and} \; \quad \frac{1}{2} \left(1 -  \frac{\sin\vt}{1+\cos\vt} \tan\vp \right).
$$
\end{theorem}

\begin{proof}
We observe that without loss of generality we may assume that \mbox{$2k\leq n$}, since otherwise, we may set $m=n-k$ and see that the functional $p_n - w p_k p_{n-k}$ remains unchanged while the new pair of integers $(m,n)$ satisfies $2m<n$. Therefore the second condition makes sense.

We will prove the necessity of the two conditions, since the sufficiency is elementary, although laborious in the case (ii). 

We assume that $c=0$. Having proved the assertion in this case we apply it to the rotated function $p(e^{-i c/n}z)$ in order to obtain the general result. 

Retracing the equalities in the proof of Theorem \ref{T1} we see that
\be
\la^n - w p_k \la^{n-k} = 1, \quad \quad \la \in supp(\mu), \label{supp1}
\ee
since equality in the triangle inequality yields constant argument and equality in the Cauchy-Schwarz inequality yields constant modulus. Formula \eqref{supp1} is equivalent to $\la^k - wp_k = \la^{k-n}$, which we integrate with respect to $\mu$ in order to get
\be
p_{n-k}= (1-2\overline{w}) \overline{p_k} = -e^{-i\vt}\overline{p_k}. \label{coe}
\ee
It is now evident that if one of the coefficients $p_k, p_{n-k}$ is zero, then both of them are zero. If $p_k=0$, case (i) clearly follows from Theorem \ref{TA}, but it can also be seen from \eqref{supp1} which becomes $\la^n=1$.

Suppose that $p_k\neq0$. In order to prove condition (ii) we begin with the additional assumption that $n=2k$. Equation \eqref{supp1} is then equivalent to \mbox{$\la^k -  \la^{-k} = wp_k$.} From this we deduce that ${\rm Im} \, \la^k$ is constant on $supp(\mu)$ and that ${\rm Re} \, (w p_k) = 0$. The former implies that for some $\zeta=e^{i\vp}$ (we may assume that $|\vp|\leq \pi /2$), the support of $\mu$ consists of the $k$-th roots of $\zeta$ and$\,-\overline{\zeta}$, having point masses, say, $m_j$ and $m_j^*$, respectively, $1\leq j \leq k$. In other words 
\be \label{s}
supp(\mu) \subseteq  e^{i \frac{\vp}{k}} U_k \cup  e^{i \frac{\pi-\vp}{k} } U_k,
\ee
with total mass in each of the two sets of the union $M=\sum_{j=1}^k m_j$ and $M^*=\sum_{j=1}^k m_j^*$, respectively. The fact that $\mu$ is a probability measure means that $M+M^*=1$. Next, we easily see that $p_k=2\int_\ST \la^k d\mu(\la)=2(\zeta M - \overline{\zeta}M^*)=2 \left( (\zeta+\overline{\zeta}) M - \overline{\zeta}  \right)$. Hence 
\begin{align} \label{mass}
0 &={\rm Re} \, (w \, p_k ) = {\rm Re} \,\left[ (1+e^{i\vt}) \left( (\zeta+\overline{\zeta}) M - \overline{\zeta}  \right)\right] \nonumber \\
& = (1+\cos\vt) \cos\vp \, (2M-1) -\sin\vp \sin\vt.
\end{align}
If $|\vp|=\pi/2$, i.e. $\!\!\!$ if $\zeta$ is either $i$ or $-i$, then $\zeta$ and $\,-\overline{\zeta}$ coincide and therefore we may choose to divide the total mass of $\mu$ into two parts in any possible way, and in particular as asserted in (ii). Otherwise, if $|\vp|<\pi/2$, equation \eqref{mass} implies 
$$
M = \frac{1}{2} \left(1 + \frac{\sin\vt}{1+\cos\vt} \tan\vp \right) .
$$
Hence, to see that \eqref{stat} has been proved, recall that we regard $U_0$ as $\ST$ and therefore, since $n=2k$, we may choose $\psi$ freely. The choice $\psi=0$ completes the proof of \eqref{stat} in case $n=2k$. 

For the remaining case $n>2k$ in the case (ii), we repeat the arguments used to prove \eqref{supp1} to get
\be
\la^n - w p_{n-k}\la^k = 1, \quad \quad \la \in supp(\mu). \label{supp2}
\ee
A combination of \eqref{supp1} and \eqref{supp2} shows that $p_k \la^{n-k}=p_{n-k}\la^k$. Hence, by \eqref{coe},
$$
\la^{n-2k} = -e^{-i\vt} \overline{p_k}/p_k, \quad \quad \la \in supp(\mu).
$$
This yields 
\be
supp(\mu) \subseteq e^{i\frac{t}{n-2k}} U_{n-2k}, \label{n-2k}
\ee
for some $t\in [0,2\pi)$. Hence $p_n = e^{it} p_{2k}, \; p_{n-k} = e^{it} p_k$ and \mbox{$2=p_n - w p_k p_{n-k}$}  $=e^{it} (p_{2k} - w p_k^2)$. It follows that the function $p(e^{it/2k}z)$ must satisfy condition \eqref{s} and, therefore, $p(z)$ satisfies the corresponding rotation of \eqref{s}. Together with \eqref{n-2k} this is
$$
supp(\mu) \subseteq (e^{i\frac{t}{n-2k}} U_{n-2k}\cap e^{i\left(\frac{\vp}{k} -\frac{t}{2k} \right)} U_k) \cup (e^{i\frac{t}{n-2k}} U_{n-2k}\cap e^{i\left(\frac{\pi-\vp}{k}-\frac{t}{2k}\right)} U_k),
$$
which is \eqref{stat} in case $c=0$. If $c\neq0$ then a further rotation by $e^{i c/n}$ and the substitution $\psi=t+c(1-2k/n)$ yield \eqref{stat}.
\end{proof}

We wish to remark that in the special case where $w=1$, Theorem \ref{T3} has the following simpler form (note that this was not explicitly stated in \cite{Li69}): It holds that $|p_n - p_k p_{n-k}|=2$ if and only if either 
\smallskip

{\rm (i)} $\;p_k=0$ and $supp(\mu) \subseteq e^{i\vp} U_n$ for some $\vp$ in $[0,2\pi)$; $\,$ or
\smallskip

{\rm (ii)} $p_k\neq0$ and
$$
supp(\mu) \subseteq (e^{i\vp} U_{n-2k}\cap e^{i\vt_1} U_k) \cup (e^{i\vp} U_{n-2k}\cap e^{i\vt_2} U_k)
$$
for some $\vp, \vt_1$ and $\vt_2$ in $[0,2\pi)$. Except for the degenerate case where the support of $\mu$ consists of only one point, the total mass of the measure in each of the two sets of the union is equal to $1/2$.

\section{Application to the self-maps of $\SD$}
There is a close connection between the class $\cP$ and self-maps of $\SD$ via conformal maps of $\SD$ to the right half plane, namely, $p=\frac{1+\vp}{1-\vp}$ is in $\cP$ for a function $\vp$ analytic in $\SD$ if and only if $\vp:\SD\to\SD$ and $\vp(0)=0$. Writing $\vp(z)=\sum_{n=1}^\infty a_n z^n$ we may relate the first few coefficients of the two functions by 
\begin{align*}
p_1 =&  2 a_1, \qquad p_2 = 2(a_2 +a_1^2), \qquad p_3 = 2(a_3 +2a_1a_2+a_1^3), \\
p_4 =& 2 (a_4 +2a_1a_3 +a_2^2 +3a_1^2a_2 +a_1^4).
\end{align*}

For functions $\vp$ of this form, the Schwarz lemma states that $|a_1|\leq 1$ while the Schwarz-Pick lemma says that $|a_2|\leq 1 -|a_1|^2$. One then easily computes 
$$
|a_2+\la a_1^2| \,\leq\, |a_2|+|\la| |a_1|^2 \,\leq\, 1 +(|\la|-1)|a_1|^2 \,\leq\, \max\{1, |\la| \}. 
$$
(See \cite{KM} for this calculation and an application of it.) The same inequality can be obtained from our Theorem \ref{T1} with $\la=1-2w$ and $n=k+1=2$. 

For higher order coefficients one has F.W. Wiener's generalization of the Schwarz-Pick lemma $|a_n|\leq 1 -|a_1|^2$ (see \cite{Boh} or problem 9 in p.172 of \cite{Ne}). However, even if we use this inequality, it does not seem easy to get the following corollary in a different way, without applying our Theorems \ref{T1} and \ref{T2}.

\begin{corollary*}
If $\vp:\SD\to\SD$ is holomorphic, $\vp(0)=0$ and $\la\in\SC$ then
\begin{align}
|a_3+(1+\la)a_1a_2+\la a_1^3| \leq & \max\{1, |\la| \} \label{self1}\\
|a_3+2\la a_1a_2 +\la^2 a_1^3| \leq & \max\{1, |\la|^2 \} \label{self2}
\end{align}
and
\begin{align}
|a_4 +(1+\la)a_1a_3 +a_2^2 +(1+2\la)a_1^2a_2 +\la a_1^4| \leq & \max\{1, |\la|\} \label{self3} \\
|a_4 +2a_1a_3 +\la a_2^2 +(1+2\la)a_1^2a_2 +\la a_1^4| \leq & \max\{1, |\la|\} \label{self4} \\
|a_4 +(1+\la)a_1a_3 +\la a_2^2 +\la(2+\la)a_1^2a_2 +\la^2 a_1^4| \leq & \max\{1, |\la|^2\} \label{self5} \\
|a_4 +2\la a_1a_3 +\la a_2^2 +3\la^2a_1^2a_2 +\la^3 a_1^4| \leq & \max\{1, |\la|^3\} \label{self6}.
\end{align}
\end{corollary*}
\begin{proof}
Set $\la=1-2w$ and apply Theorem \ref{T1} with $n=k+2=3$ to get \eqref{self1}, with $n=k+3=4$ to get \eqref{self3} and with  $n=k+2=4$ to get \eqref{self4}. Apply Theorem \ref{T2} with $k=n+1=2$ to get \eqref{self2}, with $k=n=2$ to get \eqref{self5} and with $k=n+2=3$ to get \eqref{self6}.
\end{proof}

To the best of our knowledge, these inequalities appear to be new.

\vspace{0,2cm}
\emph{Acknowledgements}. The author is supported by a fellowship of the International Program of Excellence in Mathematics at Universidad Aut\'onoma de Madrid (422Q101) and also partially supported by MINECO grant MTM2012-37436-C02-02, Spain. This work forms part of his Ph.D. thesis at UAM under the supervision of professor Dragan Vukoti\'c. The author would like to thank him for his encouragement and help. 

Moreover, the author would like to thank the referee for pointing out the simple deduction of Brown's theorem \cite{B10} and suggesting the alternative approach to Theorem \ref{T2} via the method of Delsarte and Genin \cite{DeGe}. 


\end{document}